\documentclass[11pt]{amsart}
\usepackage{amscd}
\usepackage{amssymb}
\usepackage{verbatim}
\usepackage[pdftex]{hyperref}
\usepackage{amsmath}
\usepackage{xcolor}
\usepackage{amscd}
\usepackage{amssymb}
\usepackage{amsthm}
\usepackage{verbatim}
\usepackage{multicol}
\usepackage{graphicx}
\usepackage{tikz}
\usepackage[all]{xy}
\usepackage[skip=10pt,indent=0pt]{parskip}
\usepackage{xspace}

\theoremstyle{plain}
\newtheorem{theorem}{Theorem}[section]
\newtheorem{corollary}[theorem]{Corollary}
\newtheorem{lemma}[theorem]{Lemma}

\newtheorem{proposition}[theorem]{Proposition}

\theoremstyle{definition}
\newtheorem{definition}[theorem]{Definition}

\newcommand{\N}{\hbox{${\mathbb N}$}}
\newcommand{\R}{\hbox{${\mathbb R}$}}

\newcommand{\lrvert}[1]{\left\vert #1 \right\vert}

\newcommand{\darkgreen}{green!75!black}

\title{A Pure Taxicab Perspective on Isometries}

\author[Dunbar]{Jonathan D. Dunbar}
\address{Mathematics Discipline, Saint Norbert College\\
De Pere, WI 54115}
\email{jonathan.dunbar@snc.edu}

\author[Woltman]{Nathaniel Woltman}
\address{Saint Norbert College\\
De Pere, WI 54115}
\email{nathanielwoltman02@gmail.com}

\begin{document}


\begin{abstract}
In this paper, we explicitly show the various isometries of the plane under the taxicab metric. We then use these isometries to prove that Euclid's proposition I.5 for isoscelese triangles is true under certain circumstances in taxicab geometry.
\end{abstract}

\maketitle

\section{Introduction}
%
%
%
%
Taxicab geometry dates back to 1910, with Hermann Minkowski's work in \cite{minkowski}. Though Minkowski did not use the phrase ``taxicab geometry,'' he did propose a new way of measuring distances. It was not until 1952, in \cite{menger}, when Karl Menger put a name to these taxicab distances . Later, in \cite{krause}, Krause provided an accessible introduction to taxicab geometry, exploring the metric through graphical geometric constructions and certain conic sections.  

Taxicab geometry is similar to Euclidean coordinate geometry, in that points and lines are the same. The primary difference lies in how distance is measured. The Euclidean distance between two points $P(x_1,y_1)$ and $Q(x_2,y_2)$ is commonly understood to be
\begin{equation}\label{euclideanmetric}
d_e(P,Q)=\sqrt{(x_1-x_2)^2+(y_1-y_2)^2},
\end{equation}
whereas the taxicab metric for distance is given by 
\begin{equation}\label{taxicabmetric}
d_T(P,Q)=|x_1-x_2|+|y_1-y_2|.
\end{equation}

Of particular interest to us, Thompson and Dray explored how angles, trigonometry, and other related topics behave in taxicab geometry in \cite{thompson:2000}. Then in \cite{web:taxicabgeometry}, Thompson describes two major branches of taxicab geometry: \textit{traditional} taxicab geometry and \textit{pure} taxicab geometry. Traditional taxicab geometry simply adopts the new distance metric \eqref{taxicabmetric}, but ``leaves other geometric features such as points, lines, and angles as Euclidean.'' On the other hand, ``pure taxicab geometry uses angles which are native and natural to the geometry." This paper will focus solely on pure taxicab geometry, and provide a treatment of isometries that are based in the assumption that angles and distances are those from pure taxicab geometry. Then, using these isometries, we investigate isoscelese triangles in an attempt to reconstruct a pivotal proposition from Euclid's \textit{Book I} with this new metric. We show, through the lens of pure taxicab geometry, which restrictions are necessary to preserve Euclid's proposition I.5.

\section{Background}
The adoption of the taxicab metric affects not only distance, but also how certain geometric figures, such as circles, look and are defined. In general, for some metric $d$, point $C$, and positive real number $r$, a circle is the set $\{P\ \vert\ d(P,C)=r\}$. That is, the set of all points $P$ such that the distance from $P$ to $C$ is $r$. As stated by Krause in \cite{krause}, to specify that we are talking about a taxicab circle, we need only specify that we are using the taxicab metric. In other words, a taxicab circle with center $C$ and radius $r$ is the set of all points $P$ that are a taxicab distance of $r$ away from $C$. Changing the distance metric, however, changes the appearance and behavior of the circle, as seen in Figure \ref{circle}. 

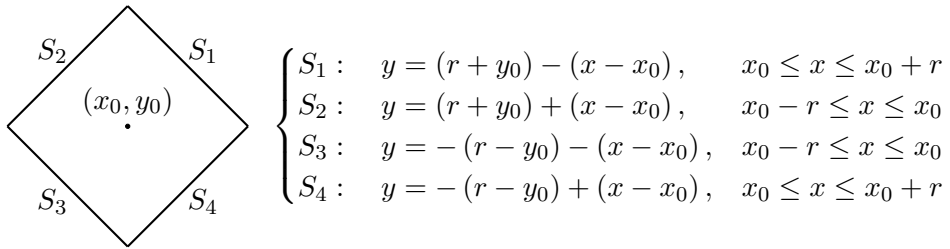
\begin{figure}[h]
\begin{tikzpicture}[scale=0.4]
    \draw[thick](0,4)--(4,0);
    \node at (2.5,2.5) {$S_1$};
    \draw[thick](0,-4)--(4,0);
    \node at (2.5,-2.5) {$S_4$};
    \draw[thick](0,4)--(-4,0);
    \node at (-2.5,2.5) {$S_2$};
    \draw[thick](0,-4)--(-4,0);
    \node at (-2.5,-2.5) {$S_3$};
    \filldraw[black] (0,0) circle (2pt);
    \node[above] at (0,0) {$(x_0,y_0)$};

    \node at (16,0) {
    $\begin{cases}
        S_1:\quad y=\left(r+y_{0}\right)-\left(x-x_{0}\right), & 
            x_{0}\le x\le x_{0}+r\\
        S_2:\quad y=\left(r+y_{0}\right)+\left(x-x_{0}\right), & 
            x_{0}-r\le x\le x_{0}\\
        S_3:\quad y=-\left(r-y_{0}\right)-\left(x-x_{0}\right), & 
            x_{0}-r\le x\le x_{0}\\
        S_4:\quad y=-\left(r-y_{0}\right)+\left(x-x_{0}\right), & 
            x_{0}\le x\le x_{0}+r\\
    \end{cases}$};
\end{tikzpicture}\caption{A taxicab circle with center $(x_0,y_0)$ and radius $r$, with labeled sides and equations.}\label{circle}
\end{figure}

Going forward, we identify side $S_i$ of a taxicab circle with center $(x_0,y_0)$ and radius $r$, using the orientation given in Figure \ref{circle}.

Notably, where angle measurements are concerned, our study of taxicab geometry may differ from other sources. We will be working with pure taxicab geometry, following \cite{thompson:2000}, wherein angle measurements are defined using the unit taxicab circle (Figure \ref{unitcircle}). Furthermore, though our focus is taxicab geometry, we will occasionally refer to Euclidean measurements. We will explicitly state we are using the Euclidean metric or the taxicab metric by adopting a subscript $e$ or $T$, respectively. Otherwise, assume that $d(P,Q)$ refers to a generic metric.

\begin{figure}[h]
\begin{tikzpicture}[scale=2]
    \draw[thick,black](0,1)--(1,0);
    \draw[thick,black](0,-1)--(1,0);
    \draw[thick,black](0,1)--(-1,0);
    \draw[thick,black](0,-1)--(-1,0);
    \draw[thick,black](-1.125,0)--(1.125,0);
    \draw[thick,black](0,-1.125)--(0,1.125);
    \draw[thick,black](-0.125,1)--(0.125,1);
    \draw[thick,black](-0.125,-1)--(0.125,-1);
    \draw[thick,black](-1,-0.125)--(-1,0.125);
    \draw[thick,black](1,-0.125)--(1,0.125);
    \node at (0.25,1){1};
    \node at (0.25,-1){-1};
    \node at (-1,-0.25){-1};
    \node at (1,-0.25){1};
\end{tikzpicture}
    \caption{The unit taxicab circle}
    \label{unitcircle}
\end{figure}
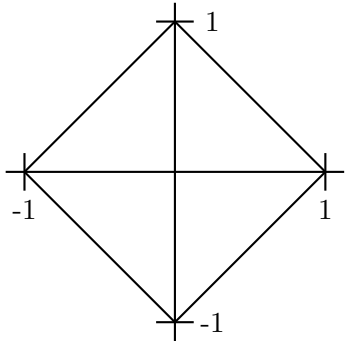

\begin{definition}\label{tradian}
A $t$-radian is an angle whose vertex is the center of a unit taxicab circle and intercepts an arc of length 1. The taxicab measure of a taxicab angle $\theta$ is the number of $t$-radians subtended by the angle on the unit taxicab circle about the vertex.
\end{definition}

In \cite{thompson:2000}, Thompson and Dray provide the above definition, as well as the subsequent theorem and lemma, which we will use. 

\begin{theorem}\label{taxicabanglestdpos}
An acute Euclidean angle $\phi_e$ in standard position has a taxicab measure of 

\begin{equation}
    \theta=2-\frac{2}{1+{tan}_e(\phi_e)}
\end{equation}
\end{theorem}

%

Using Theorem \ref{taxicabanglestdpos} and reference angles, Thompson and Dray prove the following lemma, which tells us that Euclidean right angles correspond exactly with taxicab right angles. Hence, perpendicularity in taxicab geometry will be exactly the same as perpendicularity in Euclidean geometry.

\begin{lemma}\label{rightangdef}
Euclidean right angles always have a taxicab measure of 2 $t$-radians.
\end{lemma}

%

Lines in Euclidean geometry are also lines in taxicab geometry, a fact that we have auspiciously used in Figure \ref{circle}. Since midpoints are based on the average distance in the respective $x$ and $y$ directions, and because strictly horizontal and vertical distances are equivalent in taxicab and Euclidean, the definition of midpoint is likewise the same in both geometries.

\begin{definition}\label{midpoint}
In taxicab geometry, the \textbf{midpoint} of line segment $AB$ is
$M=\left(\frac{1}{2}(x_1+x_2),\frac{1}{2}(y_1+y_2)\right)$, where $A= (x_1, y_1)$ and $B= (x_2, y_2)$.
\end{definition}

\section{Distance Preserving Maps}

Moving forward it will be useful to reduce our cases to objects oriented near the origin. However, not all of the rigid movements that we are familiar with under the Euclidean metric will maintain the distances between points under the Taxicab metric. We will begin by introducing these taxicab transformations of the 2-dimensional plane, before noting which of them preserve taxicab distances.

\begin{definition}\label{rotdef}
The rotation of a point $P$ by $r$ $t$-radians about some point $Q$ is a transformation of $P$ to some point $P'$ counterclockwise along the taxicab circle of radius $r$ centered at $Q$ such that $d_T(P,Q)=d_T(P',Q)$ and $m\angle PQP'=r$. See Figure \ref{fig:1}.
\end{definition}

\begin{definition}\label{refdef}
The reflection of a point $P$ across a line $l$ produces a point $P'$ such that $d_T(P,M)=d_T(P',M)$ and $PP'\perp l$, where $M$ is the midpoint of $PP'$ on $l$. See Figure \ref{fig:1}. 
\end{definition}

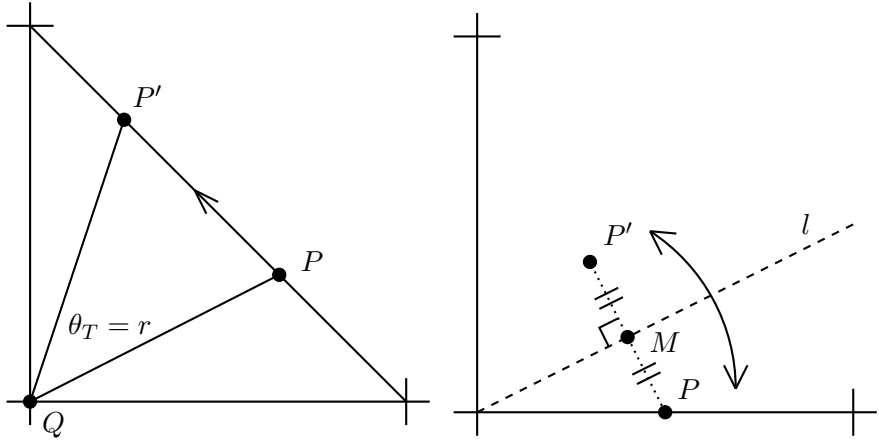
\begin{figure}[h]
\begin{minipage}{2.25in}
\begin{tikzpicture}[scale=1.25]
\draw[thick,black](-0.25,0)--(4.25,0);
\draw[thick,black](0,-0.25)--(0,4.25);
\draw[thick,black](-0.25,4)--(0.25,4);
\draw[thick,black](4,-0.25)--(4,0.25);
\draw[thick,black](0,4)--(4,0);
\draw[thick,black](0,0)--(2.65,1.35);
\draw[thick,black](0,0)--(1,3);
\filldraw [black] (2.65,1.35) circle (2pt);
\node at (3,1.5) {$P$};
\filldraw [black] (1,3) circle (2pt);
\node at (1.25,3.25) {$P'$};
\filldraw [black] (0,0) circle (2pt);
\node at (0.25,-0.25) {$Q$};
\node at (0.85,0.8) {${\theta}_T=r$};
\draw[thick,black](1.75,2.25)--(1.875,2);
\draw[thick,black](1.75,2.25)--(2,2.125);
\end{tikzpicture}
\end{minipage}\hspace{0.1cm}
\begin{minipage}{2.25in}
\begin{tikzpicture}[scale=1.25]
\draw[thick,black](-0.25,0)--(4.25,0);
\draw[thick,black](0,-0.25)--(0,4.25);
\draw[thick,black](-0.25,4)--(0.25,4);
\draw[thick,black](4,-0.25)--(4,0.25);
\draw[thick,black,dashed](0,0)--(4,2);
\node at (3.5,2) {$l$};
\draw[thick,black,dotted](2,0)--(1.2,1.6);
\draw[thick,black] (2.75,0.25) arc
    [
        start angle=0,
        end angle=57,
        x radius=2,
        y radius=2
    ] ;
\draw[thick,black](2.75,0.25)--(2.625,0.5);
\draw[thick,black](2.75,0.25)--(2.875,0.5);
\draw[thick,black](1.839,1.927)--(2.119,1.902);
\draw[thick,black](1.839,1.927)--(1.964,1.677);
\filldraw [black] (2,0) circle (2pt);
\node at (2.25,0.25) {$P$};
\filldraw [black] (1.2,1.6) circle (2pt);
\node at (1.5,1.9) {$P'$};
\filldraw [black] (1.6,0.8) circle (2pt);
\node at (2,0.75) {$M$};
\draw[thick,black](1.4,0.7)--(1.3,0.9)--(1.5,1);
\draw[thick,black](1.3,1.1)--(1.55,1.225);
\draw[thick,black](1.25,1.2)--(1.5,1.325);
\draw[thick,black](1.7,0.3)--(1.95,0.425);
\draw[thick,black](1.65,0.4)--(1.9,0.525);
\end{tikzpicture}
\end{minipage}\caption{An example of a taxicab rotation and reflection}
\label{fig:1}
\end{figure}

\begin{lemma}\label{twotrotations}
    A Euclidean rotation by $\frac{\pi}{2}$ radians, and a taxicab rotation by 2 $t$-radians are equivalent transformations of $\R^2$.
\end{lemma}

\begin{proof}
    From Lemma \ref{rightangdef}, we know that angles measuring 2 $t$-radians and $\frac{\pi}{2}$ radians are equivalent. Let $P(x,y)$ be a point in the plane, with $\lrvert{x}+\lrvert{y}=r$. Hence, $P$ is a taxicab distance of $r$ from the origin and therefore lies on the taxicab circle of radius $r$ centered at the origin, which we will call $C$. Note that we may write point $P(x,y)$. The Euclidean rotation of $P$ by $\frac{\pi}{2}$ radians maps to the point $P'(-y,x)$. Observe that $P'$ is also a taxicab distance of $r$ from the origin since $\lrvert{-y}+\lrvert{x}=r$, and thus $P'$ is also on $C$.
    
    We wish to show that the distance along $C$ from $P$ to $P'$ is $2r$, and hence $P'$ is the result of a taxicab rotation but 2 $t$-radians. Assume the case that $P$ is in quadrant 1 and $P'$ is in quadrant 2, as all other cases are similarly shown. By this assumption, $x,y>0$, $y=r-x$, $P(x,y)=(x,r-x)$, and $P'(-y,x) = (x-r,x)$. 
    
    Let point $V$ be $(0,r)$. Rotating from $P$ to $P'$ along $C$ requires passing through $V$. Then, $d_T(P,V) = \lrvert{x}+\lrvert{r-(r-x)} = 2\lrvert{x}$ and $d_T(V,P) = \lrvert{r-x}+\lrvert{r-x} = 2\lrvert{r-x}$. So, the distance along $C$ is $2\lrvert{x}+2\lrvert{r-x} = 2\lrvert{x} + 2\lrvert{y} = 2r$, and thus $P'$ is a taxicab rotation of 2 $t$-radians counterclockwise from $P$.
\end{proof}

\begin{lemma}\label{y=x}
    Euclidean and taxicab reflections across the line $y=x$ are equivalent transformations of $\R^2$, as are Euclidean and taxicab reflections across the line $y=-x$, $y=0$, $x=0$.
\end{lemma}
\begin{proof}
    The result of the Euclidean reflection of the point $P(x,y)$ across the line $y=x$ is $Q(y,x)$. 

    Assume we taxicab reflect the point $P(x,y)$ across the line $y=x$ resulting in the point $P'(x',y')$. By, Definition \ref{refdef}, $PP'$ is perpendicular to the line $y=x$ and thus has slope $-1$, as it would in Euclidean geometry. Additionally by \ref{refdef}, the midpoint $M$ of $PP'$ will fall on the line $y=x$ and so we may label it $M(a,a)$. Note that, by Definition \ref{midpoint}, $a=\frac{1}{2}(x+x')$ and $a=\frac{1}{2}(y+y')$. Then, since the slope of $PM$ is $-1$ and the slope of $P'M$ is $-1$, we know that $a=\frac{1}{2}(x+y)$ and $a=\frac{1}{2}(x'+y')$. By these four equations, $x'=y$ and $y'=x$. Hence, $P'(x',y') = Q(y,x)$, and we have that Euclidean and taxicab reflections across the line $y=x$ are equivalent.

    The proof that Euclidean and taxicab reflections across the line $y=-x$ (and $y=0$, $x=0$, respectively) are equivalent is similarly shown.
\end{proof}

\begin{definition}\label{distancepresmap}
A taxicab isometry is a map $f:\R^2\longrightarrow\R^2$ such that $d_T(P,Q)=d_T(f(P),f(Q))$, for all points $P$, $Q$ in $\R^2$.
\end{definition}

\begin{theorem}\label{finitecombo}
A composition of any finite number of taxicab isometries is also a taxicab isometry.
\end{theorem}

\begin{proof}
Let $f,g: \R^2\longrightarrow\R^2$ be taxicab isometries and let $P,Q$ be two points where $P,Q\in\R^2$. Then,
\begin{align*}
    d_T((f\circ g)(P),(f\circ g)(Q))&=d_T(f(g(P)),f(g(Q)))\\
    &=d_T(g(P),g(Q))
    =d_T(P,Q).
\end{align*}

Inductively, we can show that this is true for any finite number of isometries. Therefore, the composition of any finite number of taxicab isometries is also an isometry.
\end{proof} 

In this paper, we use a collection of taxicab isometries and their compositions to narrow our focus and reduce the number of cases we examine to figures situated near the origin. Explicitly, this will include taxicab circles centered at the origin and angles contained within one, two, or three quadrants with their vertex located at the origin. Once we show a property is true for one of these objects, then it will be true for isometric objects located elsewhere in the plane. Our first requirement, then, is to show that translations are always taxicab isometries.

\begin{lemma}\label{trans}
Translations are taxicab isometries.
\end{lemma}
\begin{proof}
Let points $A = (x_1,y_1)$ and $B= (x_2,y_2)$ in $\R^2$. Additionally, let $f$ be the translation defined by $f(x,y) = (x+a,y+b)$ for some real numbers $a$, $b$. Then, $f(A) = (x_1+a,y_1+b)$ and $f(B) = (x_2+a,y_2+b)$, and
\begin{align*}
    d_T(f(A),f(B))& 
    =|(x_1+a)-(x_2+a)|+|(y_1+b)-(y_2+b)|\\
    &=|x_1-x_2|+|y_1-y_2|
    =d_T(A,B)
\end{align*}

Therefore, translations are taxicab isometries.
\end{proof}

As Thompson mentions in \cite{thompson:2000}, ``although angles are translation invariant, they are not rotation invariant." Therefore, the taxicab measure of a Euclidean angle depends on its orientation within the plane, and the same Euclidean angle may have two different taxicab measures depending on its rotation. However, as we will show, there are certain special rotations and reflections in taxicab geometry that preserve distances.

\begin{lemma}\label{rot2}
Rotations of 2 $t$-radians counterclockwise about the origin are taxicab isometries.
\end{lemma}
\begin{proof}
Let points $A=(x_1,y_1)$ and $B=(x_2,y_2)$. Additionally, let $f$ be the rotation of $\R^2$ about the origin by 2 $t$-radians. By Lemma \ref{rightangdef}, taxicab right angles are precisely Euclidean right angles, so $2$ $t$-radians $=\frac{\pi}{2}$ radians. We know that in Euclidean geometry, any point $P(x,y)$ maps to the point $P'(-y,x)$ when rotated $\frac{\pi}{2}$ radians about the origin. So, $f$ maps point $A$ to point $A'(-y_1,x_1)$ and point $B$ to point $B'(-y_2,x_2)$. Thus,


\begin{align*}
    d_T(f(A),f(B))& = d_T(A',B') =|-y_1-(-y_2)|+|x_1-x_2|\\
    &=|y_1-y_2|+|x_1-x_2| =d_T(A,B)
\end{align*}

 Additionally,  Therefore, rotating points $\frac{\pi}{2}$ radians, or 2 $t$-radians, counterclockwise about the origin is a taxicab isometry.
\end{proof}

\begin{corollary}
For all $n\in\N$, rotations of $2n$ $t$-radians counterclockwise about the origin are taxicab isometries.
\end{corollary}

This result follows quite naturally from \ref{finitecombo} and \ref{rot2}.

\begin{lemma}\label{refy=x}
Reflections across the lines $y=x$, $y=-x$, $y=0$, and $x=0$ are taxicab isometries.
\end{lemma}
\begin{proof}
Let point $A$ be the point $(x_1,y_1)$ and point $B$ be the point $(x_2,y_2)$. Additionally, let $f$ be a reflection of $\R^2$ across the line $y=x$. By \ref{y=x}, we know taxicab reflections across these lines are exactly Euclidean reflections. 

In Euclidean geometry, reflections across the line $y=x$ map any point $P(x,y)$ to the point $P'(y,x)$. So, $f$ maps point $A$ to point $A'(y_1,x_1)$ and point $B$ to point $B'(y_2,x_2)$. Thus,

\begin{align*}
    d_T(f(A),f(B))&=d_T(A',B') =|y_1-y_2|+|x_1-x_2|\\
    &=|x_1-x_2|+|y_1-y_2| =d_T(A,B)
\end{align*}

 Therefore, reflections across the line $y=x$ are taxicab isometries. Because reflections across these lines $y=-x$, $y=0$, and $x=0$ are the same in taxicab and Euclidean geometries, we also have the following Euclidean mapping rules on points:
    \begin{itemize}
        \item   The image of point $P(x,y)$, reflected across $y=-x$ is $P'(-y,-x)$.
        \item   The image of point $P(x,y)$, reflected across $y=0$ is $P'(x,-y)$.
        \item   The image of point $P(x,y)$, reflected across $x=0$ is $P'(-x,y)$.
    \end{itemize}
With these, it is straightforward to show that taxicab distances are preserved by these reflections, similar to the proof for reflections across the line $y=x$.
\end{proof}

The lemmas above, with Theorem \ref{finitecombo}, coalesce into the theorem below.
\begin{theorem}\label{refrottrans}
    Any combination of a finite number of these taxicab isometries is also a taxicab isometry:
    \begin{itemize}
        \item   translations in the plane,
        \item   rotations by $2n$ $t$-radians (for all integers $n$),
        \item   reflections across the lines $y=x$, $y=-x$, $y=0$, $x=0$.
    \end{itemize}
\end{theorem}

\section{The Nature of Isosceles Triangles in Pure Taxicab Geometry}
A triangle is called isosceles if it has two sides of equivalent length. Euclid's proposition I.5 transfers the property of side length equivalence to the angles opposite of those sides. Heath's translation is as follows (from \cite{heath}).

\begin{theorem}[Euclid's Proposition I.5]
In isosceles triangles the angles at the base are equal to one another and, if the equal straight lines be produced further, the angles under the base will be equal to one another.
\end{theorem}


Using Theorem \ref{refrottrans}, we will view an isosceles triangle as a pair of radii of a taxicab circle centered at the origin, such that the vertex of the angle formed by those two sides is the center of the circle, and the third side is the line segment between the points where the radii meet the sides of the circle. Henceforth, in $\R^2$, when we refer to \textit{adjacent quadrants}, we will mean a pair of quadrants that share an axis, while \textit{opposing quadrants} will be a pair that does not. Furthermore, we shall consider any part of an axis as included in each of the quadrants to which it is adjacent. 

\begin{proposition}[I.5T]\label{i5t}
Suppose we have triangle $\triangle PQO$, where $PO\cong QO$ and $O$ is the origin.
\begin{enumerate}
    \item   If $P$ and $Q$ are in the same quadrant, then $\angle PQO\cong\angle QPO$ if and only if the horizontal component of $P$ equals the vertical component of $Q$, and the horizontal component of $Q$ equals the vertical component of $P$. 
    \item   If $P$ and $Q$ are in adjacent quadrants, then $\angle PQO\cong\angle QPO$ if and only if $P$ and $Q$ have equivalent horizontal components and equivalent vertical components. 
    \item   If $P$ and $Q$ are in opposing quadrants, then $\angle PQO\cong\angle QPO$.
\end{enumerate}
\end{proposition}

\begin{proof}
Let $\triangle PQO$ be a triangle in $\R^2$ where $O$ is the origin and $PO\cong QO$, $\angle POQ=\theta$ with $\theta<4$ t-radians. Observe that points $P,Q$ both lie on the taxicab circle with radius $r\ne0$ centered at the origin, and radii $OP,OQ$ have length $r$. We assume $\triangle PQO$ is in one of these three orientations, for if it is not then some combination of isometries listed in Theorem \ref{refrottrans} will result in one of these cases without altering distances: 
\begin{enumerate}
    \item   $\theta$ is completely contained within quadrant I.
    \item   $\theta$ is split across quadrants I and II.
    \item   $\theta$ is split across quadrants I, II, and III.
\end{enumerate}
    
    Suppose $\theta$ is completely contained within quadrant I, and so let $P(x_1,y_1)$ and $Q(x_0,y_0)$ lie within quadrant I, where $x_0,x_1,y_0,y_1\in\R$, $x_0,x_1,y_0,y_1\geq0$, and $x_1>x_0,y_1>y_0$. See Figure \ref{fig:2} below, where equivalent angles can be identified using Theorem \ref{refrottrans} and angle measurements are in $t$-radians. 
    
    \begin{figure}[h]\begin{tikzpicture}[scale=1.667]
    \draw[thick,black](-0.25,0)--(4.25,0);
    \draw[thick,black](0,-0.25)--(0,4.25);
    \draw[thick,black](4,-0.25)--(4,0.25);
    \draw[thick,black](-0.25,4)--(0.25,4);
    \draw[thick,black](0,4)--(1.5,2.5);
    \draw[thick,black](3,1)--(4,0);
    \draw[thick,red](0,0)--(3,1);
    \node[red] at (1.833,0.45) {$r$}; 
    \draw[thick,blue](0,0)--(1.5,2.5);
    \node[blue] at (0.667,1.5) {$r$}; 
    \draw[thick,\darkgreen](1.5,2.5)--(3,1);
    \node[\darkgreen] at (2.375,2) {$c$};
    \filldraw [black] (1.5,2.5) circle (2pt);
    \node at (2,2.867) {$P(x_1,y_1)$};
    \node[red] at (1.625,2.125) {$\beta$};
    \draw[thick,red] (1.2,2) arc
    [
        start angle=240,
        end angle=315,
        x radius=0.583cm,
        y radius=0.583cm
    ] ;
    \filldraw [black] (3,1) circle (2pt);
    \node at (3.5,1.333) {$Q(x_0,y_0)$};
    \node[blue] at (2.625,1.125) {$\alpha$};
    \draw[thick,blue] (2.6,1.4) arc
    [
        start angle=135,
        end angle=198,
        x radius=0.566cm,
        y radius=0.566cm
    ] ;
    \node[\darkgreen] at (0.4,0.35) {{\Large $\theta$}};
    \draw[thick,black,dashed](1.5,2.5)--(1.5,0);
    \draw[thick,black](1.75,0.25)--(1.5,0.25); 
    \draw[thick,black](1.75,0.25)--(1.75,0); 
    \draw[thick,black,dashed](3,1)--(0,1); 
    \draw[thick,black,dashed](0,2.5)--(3,2.5);
    \draw[thick,black,dashed](3,2.5)--(3,0);
    \draw[thick,black](0,1.25)--(0.25,1.25); 
    \draw[thick,black](0.25,1.25)--(0.25,1); 
    \draw[thick,black](0,2.25)--(0.25,2.25); 
    \draw[thick,black](0.25,2.25)--(0.25,2.5); 
    \draw[thick,black](2.75,0.25)--(3,0.25); 
    \draw[thick,black](2.75,0.25)--(2.75,0); 
    \node at (1.5,-0.125) {$x_0$};
    \node at (-0.125,1.25) {$y_1$};
    \node at (3.5,-0.125) {$y_0$};
    \node at (-0.125,3.25) {$x_1$};
    \node at (0.75,2.625) {$x_1$};
    \node at (3.125,0.5) {$y_0$};
    \draw[thick, \darkgreen, domain=18:60] plot ({0.9*cos(\x)},{0.9*sin(\x)});
    \draw[thick,orange] (1,0) arc
    [
        start angle=0,
        end angle=18,
        x radius=1cm,
        y radius=1cm
    ] ;
    \draw[thick,orange] (2,1) arc
    [
        start angle=180,
        end angle=198,
        x radius=1cm,
        y radius=1cm
    ] ;
    \draw[thick,violet] (.4,.667) arc
    [
        start angle=60,
        end angle=90,
        x radius=.777cm,
        y radius=.777cm
    ] ;
    \draw[thick,violet] (1,1.667) arc
    [
        start angle=240,
        end angle=270,
        x radius=.972cm,
        y radius=.972cm
    ] ;
    \node[orange] at (0.75,0.125) {$\gamma$};
    \node[orange] at (2.25,0.875) {$\gamma$};
    \node[violet] at (0.125,0.5) {$\phi$};
    \node[violet] at (1.25,1.75) {$\phi$};
    \draw[thick,black] (1.5,1.5) arc
    [
        start angle=270,
        end angle=315,
        x radius=1cm,
        y radius=1cm
    ] ;
    \draw[thick,black] (2.25,1.75) arc
    [
        start angle=135,
        end angle=181,
        x radius=1cm,
        y radius=1cm
    ] ;
    \node at (2.25,1.25) {1};
    \node at (1.75,1.75) {1};
    \filldraw[black] (0,0) circle (2pt);
    \node at (-0.5,-0.25) {$O(0,0)$};
\end{tikzpicture}\caption{$\theta$ is completely contained within quadrant I}
\label{fig:2}
\end{figure}

Additionally, notice that $r = x_1+y_1 = x_0+y_0$. Thus, using Definition \ref{taxicabanglestdpos}

\begin{minipage}{2in}
\begin{align*}
    \alpha&=\gamma+1\\
    &=2-\frac{2}{1+\text{tan}_e(\gamma_e)}+1\\
    &=3-\frac{2}{1+\frac{y_0}{x_0}}\\
    &=3-\frac{2x_0}{r}
\end{align*}
\end{minipage}\hspace{2cm}
\begin{minipage}{2in}
\begin{align*}
    \beta&=\phi+1\\
    &=2-\frac{2}{1+\text{tan}_e(\phi_e)}+1\\
    &=3-\frac{2}{1+\frac{x_1}{y_1}}\\
    &=3-\frac{2y_1}{r}.
\end{align*}
\end{minipage}

We see that $\alpha=\beta$ only when $x_0=y_1$. Furthermore, $0 =x_0-y_1 =x_0-r+r-y_1 =-y_0+x_1$ implies that $x_1 = y_0$.
Therefore, in the case where $\theta$ is completely contained within quadrant I, the base angles of isosceles $\triangle PQO$, $\angle PQO$ and $\angle QPO$ are equal only when $x_0=y_1$ and $x_1=y_0$, where $P=(x_0,y_0)$ and $Q(x_1,y_1)$.

Now, suppose $\theta$ is split across quadrants I and II. Without loss of generality, assume $y_0\geq y_1$ and $P=(x_0,y_0)$ and $Q=(-x_1,y_1)$. Then,  $x_0,x_1,y_0,y_1\in\R$, $x_0,x_1,y_0,y_1\geq0$, $x_1\geq x_0$. See Figure \ref{fig:3} below, where equivalent angles can be identified using Theorem \ref{refrottrans} and angle measurements are in $t$-radians.
    
    \begin{figure}[h]
       \begin{tikzpicture}[scale=1.4]
    \draw[thick,black](-4.25,0)--(4.25,0);
    \draw[thick,black](0,-0.25)--(0,4.25);
    \draw[thick,black](-4,-0.25)--(-4,0.25);
    \draw[thick,black](4,-0.25)--(4,0.25);
    \draw[thick,black](-0.25,4)--(0.25,4);
    \draw[thick,black](0,4)--(4,0);
    \draw[thick,black](0,4)--(-4,0);
    \draw[thick,red](0,0)--(1.5,2.5);
    \node[red] at (0.875,1.2) {$r$}; 
    \draw[thick,blue](0,0)--(-3,1); 
    \node[blue] at (-1.625,0.375) {$r$}; 
    \draw[thick,\darkgreen](-3,1)--(1.5,2.5);
    \node[\darkgreen] at (-0.75,2) {$c$};
    \filldraw [black] (1.5,2.5) circle (2pt);
    \node at (2,2.867) {$P(x_0,y_0)$};
    \filldraw [black] (-3,1) circle (2pt);
    \node at (-3.65,1.333) {$Q(-x_1,y_1)$};
    \draw[thick,black,dashed](-3,0)--(-3,2.5)--(1.5,2.5)--(1.5,0);
    \draw[thick,black,dashed](-3,1)--(3,1);
    \node at (0.75,-0.125) {$x_0$};
    \node at (1.625,1.25) {$y_0$};
    \node at (-1.5,-0.125) {$x_1$};
    \node at (-3.125,0.5) {$y_1$};
    \node at (-3.5,-0.125) {$y_1$};
    \node at (2.75,-0.125) {$y_0$};
    \draw[thick, \darkgreen, domain=59:162] plot ({0.9*cos(\x)},{0.9*sin(\x)});
    \draw[thick,violet] (0.7,0) arc
    [
        start angle=0,
        end angle=59,
        x radius=0.7cm,
        y radius=0.7cm
    ] ;
    \draw[thick,violet] (0.125,2.5) arc
    [
        start angle=180,
        end angle=239,
        x radius=1.375cm,
        y radius=1.375cm
    ] ;
    \draw[thick,orange] (-0.9,0.3) arc
    [
        start angle=162,
        end angle=180,
        x radius=0.949cm,
        y radius=0.949cm
    ] ;
    \draw[thick,orange] (-2.1,0.7) arc
    [
        start angle=342,
        end angle=360,
        x radius=0.949cm,
        y radius=0.949cm
    ] ;
    \draw[thick,yellow] (-2,1) arc
    [
        start angle=0,
        end angle=18,
        x radius=1cm,
        y radius=1cm
    ] ;
    \draw[thick,yellow] (0.5,2.5) arc
    [
        start angle=180,
        end angle=198,
        x radius=1cm,
        y radius=1cm
    ] ;
    \draw[thick,blue] (0.6,2.2) arc
    [
        start angle=198,
        end angle=239,
        x radius=0.949cm,
        y radius=0.949cm
    ] ;
    \draw[thick,red] (-1.5,0.5) arc
    [
        start angle=342,
        end angle=378,
        x radius=1.581cm,
        y radius=1.581cm
    ] ;
    \node[violet] at (0.5,1.875) {$\phi$};
    \node[violet] at (0.375,0.25) {$\phi$};
    \node[orange] at (-0.75,0.125) {$\gamma$};
    \node[orange] at (-2.25,0.875) {$\gamma$};
    \node[yellow] at (-2.25,1.125) {$\lambda$};
    \node[yellow] at (0.6875,2.375) {$\lambda$};
    \node[red] at (-1.75,1.125) {$\beta$};
    \node[blue] at (0.875,2) {$\alpha$};
    \node[\darkgreen] at (-0.25,0.375) {$\theta$};
    \draw[thick,black](1.25,0)--(1.25,0.25)--(1.5,0.25);
    \draw[thick,black](-2.75,0)--(-2.75,0.25)--(-3,0.25);
    \draw[thick,black](-0.25,2.5)--(-0.25,2.25)--(0,2.25);
    \draw[thick,black](-0.25,1)--(-0.25,1.25)--(0,1.25);
    \filldraw[black] (0,0) circle (2pt);
    \node at (-0.5,-0.25) {$O(0,0)$};
    \end{tikzpicture}\caption{$\theta$ is split across quadrants I and II}
    \label{fig:3}
    \end{figure}

Observe that $r=x_0+y_0=x_1+y_1$, $c = x_1 + x_0 + y_0 - y_1$ and 
\begin{align*}
    \alpha =\phi-\lambda 
    &=2-\frac{2}{1+\text{tan}_e(\phi_e)}-\left(2-\frac{2}{1+\text{tan}_e(\lambda_e)}\right)\\
    &=\frac{2}{1+\frac{y_0-y_1}{x_0+x_1}}-\frac{2}{1+\frac{y_0}{x_0}}\\
    &=\frac{2(x_0+x_1)}{y_0-y_1+x_0+x_1}-\frac{2x_0}{x_0+y_0}\\
    &=\frac{2(x_0+x_1)}{c}-\frac{2x_0}{r}
\end{align*}

\begin{align*}
    \beta = \gamma+\lambda &=2-\frac{2}{1+\text{tan}_e(\gamma_e)}+2-\frac{2}{1+\text{tan}_e(\lambda_e)}\\
    &=4-\frac{2}{1+\frac{y_1}{x_1}}-\frac{2}{1+\frac{y_0-y_1}{x_0+x_1}}\\
    &=4-\frac{2x_1}{x_1+y_1}-\frac{2(x_0+x_1)}{x_0+x_1+y_0-y_1}\\
    &=4-\frac{2(x_0+x_1)}{c}-\frac{2x_1}{r}.
\end{align*}

We wish to know the conditions for which $\beta - \alpha = 0$. Note that $c\ne0$ and since $r=x_0+y_0 = x_1+y_1$ then $c = x_1 + x_0 + y_0 - y_1 = 2x_1$.

\begin{align*}
    0 = \beta - \alpha
    &=4-\frac{2(x_0+x_1)}{c}-\frac{2x_1}{r}-\left(\frac{2(x_0+x_1)}{c}-\frac{2x_0}{r}\right)\\
    &=4-\frac{4r(x_0+x_1)+2c(x_0-x_1)}{cr}\\
    &=4cr-4r(x_0+x_1)+2c(x_0-x_1)\\
    &=2cr-2r(x_0+x_1)+c(x_0-x_1)\\
    &=2(x_1+y_1)(2x_1)-2(x_1+y_1)(x_0+x_1)+(2x_1)(x_0-x_1)\\
    &=(x_1+y_1)(2x_1)-(x_1+y_1)(x_0+x_1)+(x_1)(x_0-x_1)\\
    &=2{x_1}^2+2x_1y_1-x_0x_1-{x_1}^2-x_0y_1-x_1y_1+x_0x_1-{x_1}^2\\
    &=(x_1-x_0)y_1\\
\end{align*}

We see that $\alpha=\beta$ is only true when $y_1=0$ or $x_0=x_1$. If $y_1 = 0$, then $Q$ is on an axis at point $(-r,0)$ and is handled by case 3. So, for now, we only need to consider when $x_0 = x_1$. Notice that $0 = x_1-x_0 = x_1-r+r-x_0 = y_0-y_1$ and thus $x_0=x_1$ if and only if $y_0=y_1$.


Therefore, in the case where $\theta$ is split across quadrants I and II, the base angles of isosceles $\triangle PQO$, $\angle PQO$ and $\angle QPO$ are equal only when when $x_0=x_1$ and $y_0=y_1$, where $P=(x_0,y_0)$ and $Q(-x_1,y_1)$.

Lastly, suppose $\theta$ is split across quadrants I, II, and III. For clarity and conciseness, assume $P(x_0,y_0)$ lies in quadrant I and $Q(-x_1,-y_1)$ lies in quadrant III, where $x_0,x_1,y_0,y_1\in\R$, $x_0,x_1,y_0,y_1\geq0$, $x_1>x_0$, and $y_0>y_1$. See Figure \ref{fig:4} below.

\begin{figure}[h]
\begin{tikzpicture}[scale=1.4]
    \draw[thick,black](-4.25,0)--(4.25,0);
    \draw[thick,black](0,-4.25)--(0,4.25);
    \draw[thick,black](-4,-0.25)--(-4,0.25);
    \draw[thick,black](4,-0.25)--(4,0.25);
    \draw[thick,black](-0.25,-4)--(0.25,-4);
    \draw[thick,black](-0.25,4)--(0.25,4);
    \draw[thick,black](0,4)--(4,0);
    \draw[thick,black](0,-4)--(4,0);
    \draw[thick,black](0,4)--(-4,0);
    \draw[thick,black](0,-4)--(-4,0);
    \draw[thick,red](0,0)--(1.5,2.5);
    \node[red] at (0.85,1.1) {$r$}; 
    \draw[thick,blue](0,0)--(-3,-1);
    \node[blue] at (-1.5,-0.65) {$r$}; 
    \draw[thick,\darkgreen](-3,-1)--(1.5,2.5);
    \node[\darkgreen] at (-0.75,1) {$c$};
    \filldraw [black] (1.5,2.5) circle (2pt);
    \node at (2,2.867) {$P(x_0,y_0)$};
    \filldraw [black] (-3,-1) circle (2pt);
    \node at (-3.5,-1.333) {$Q(-x_1,-y_1)$};
    \draw[thick,black,dashed](-3,-1)--(-3,2.5)--(1.5,2.5)--(1.5,-1)--(-3,-1);
    \node at (-1.5,-1.125) {$x_1$};
    \node at (.75,-1.125) {$x_0$};
    \node at (.75,0.125) {$x_0$};
    \node at (-0.125,-0.5) {$y_1$};
    \node at (1.625,-0.5) {$y_1$};
    \node at (1.625,1.25) {$y_0$};
    \draw[thick,violet] (0.5,0.833) arc
    [
        start angle=59,
        end angle=90,
        x radius=0.972cm,
        y radius=0.972cm
    ] ;
    \draw[thick,violet] (1,1.667) arc
    [
        start angle=239,
        end angle=270,
        x radius=0.972cm,
        y radius=0.972cm
    ] ;
    \draw[thick,orange] (-1,0) arc
    [
        start angle=180,
        end angle=198,
        x radius=1cm,
        y radius=1cm
    ] ;
    \draw[thick,orange] (-2,-1) arc
    [
        start angle=0,
        end angle=18,
        x radius=1cm,
        y radius=1cm
    ] ;
    \draw[thick,red] (-2.1,-0.7) arc
    [
        start angle=18,
        end angle=38,
        x radius=.947cm,
        y radius=.947cm
    ] ;
    \draw[thick,blue] (0.78,1.94) arc
    [
        start angle=218,
        end angle=239,
        x radius=0.912cm,
        y radius=0.912cm
    ] ; 
    \draw[thick,\darkgreen] (0.3,0.5) arc
    [
        start angle=61,
        end angle=198,
        x radius=0.583cm,
        y radius=0.583cm
    ] ; 
    \node[red] at (-2.3,-0.625) {$\beta$};
    \node[blue] at (1,1.925) {$\alpha$};
    \node[\darkgreen] at (-0.25,0.25) {$\theta$};
    \node[orange] at (-0.75,-0.125) {$\gamma$};
    \node[orange] at (-2.25,-0.875) {$\gamma$};
    \node[violet] at (0.1875,0.75) {$\phi$};
    \node[violet] at (1.3125,1.75) {$\phi$};
    \draw[thick,black](1.25,0)--(1.25,0.25)--(1.5,0.25);
    \draw[thick,black](-2.75,0)--(-2.75,0.25)--(-3,0.25);
    \draw[thick,black](0,-0.75)--(-0.25,-0.75)--(-0.25,-1);
    \draw[thick,black](0,2.25)--(-0.25,2.25)--(-0.25,2.5);
    \filldraw[black] (0,0) circle (2pt);
    \node at (0.5,-0.25) {$O(0,0)$};
\end{tikzpicture}\caption{$\theta$ is split across quadrants I, II, and III}
\label{fig:4}
\end{figure}
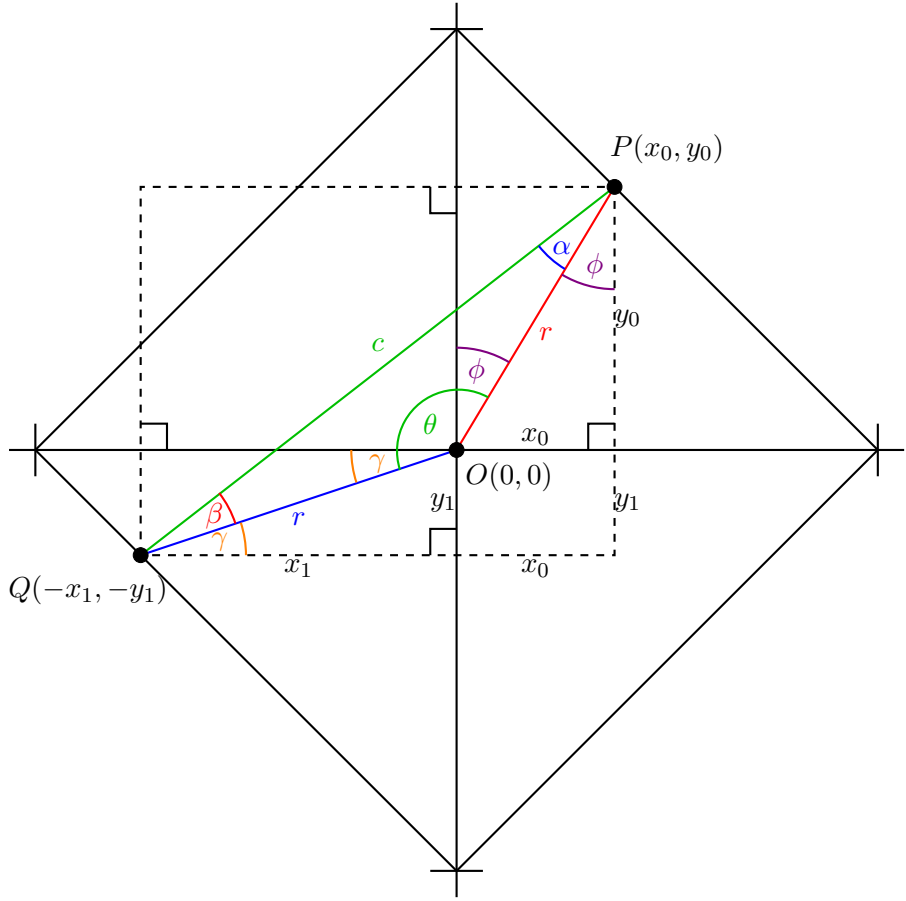

Thus, we follow the same strategy as earlier cases.
\begin{align*}
    \alpha  &=\frac{2}{1+\text{tan}_e(\phi_e)}-\frac{2}{1+\text{tan}_e(\phi_e+\alpha_e)}\\
    &=\frac{2}{1+\frac{x_0}{y_0}}-\frac{2}{1+\frac{x_0+x_1}{y_0+y_1}}\\
    &=\frac{2y_0}{x_0+y_0}-\frac{2(y_0+y_1)}{y_0+y_1+x_0+x_1}\\
    &=\frac{2y_0}{r}-\frac{y_0+y_1}{r}\\
    &=\frac{y_0-y_1}{r}
\end{align*}
\begin{align*}
    \beta   &=\frac{2}{1+\text{tan}_e(\gamma_e)}-\frac{2}{1+\text{tan}_e(\gamma_e+\beta_e)}\\
    &=\frac{2}{1+\frac{y_1}{x_1}}-\frac{2}{1+\frac{y_0+y_1}{x_0+x_1}}\\
    &=\frac{2x_1}{x_1+y_1}-\frac{2(x_0+x_1)}{x_0+x_1+y_0+y_1}\\
    &=\frac{2x_1}{r}-\frac{x_0+x_1}{r}\\
    &=\frac{x_1-x_0}{r}
\end{align*}

We look for when $\alpha-\beta=0$. However, regardless of the values of $P(x_0,y_0)$ and $Q(-x_1,-y_1)$, $\alpha-\beta=\frac{y_0-y_1}{r}-\frac{x_1-x_0}{r} = \frac{(x_0+y_0)-(x_1+y_1)}{r} =\frac{r-r}{r} = 0$.

Therefore, in the case where $\theta$ is split across quadrants I, II, and III, the base angles of isosceles $\triangle PQO$, $\angle PQO$ and $\angle QPO$, are always equal, where $P=(x_0,y_0)$ and $Q(-x_1,-y_1)$, including when one or more of these points lies on an axis.

Thus, for isosceles triangle $\triangle PQO$ where $O$ is at the origin and $PO \cong QO$, we have corresponding base angle equivalence only in the cases outlined in the above proposition.
\end{proof}

By directly applying Theorem \ref{refrottrans}, we obtain the following result.

\begin{corollary}
    Any isosceles triangle $\triangle ABC$ where $AC \cong BC$ will have corresponding base angle equivalence only when its isometric image $\triangle PQO$ satisfies one of the conditions in Proposition \ref{i5t}.
\end{corollary}


\bibliography{Taxicab} 

\bibliographystyle{plain} 

\end{document}